\documentclass[12pt]{amsart}
\usepackage{amssymb}
\usepackage{tikz}
\usepackage{enumerate}
\usepackage{enumitem}
\usepackage{bbm}
\usepackage{mathtools}
\usepackage[colorlinks,linkcolor=blue,citecolor=magenta]{hyperref}
\usepackage{color}
\usepackage{fullpage}
\usepackage{tikz}
\usepackage{subcaption}
\usetikzlibrary{calc}
\usetikzlibrary{intersections}
\usepackage{fullpage}
\usepackage{adjustbox}

\usepackage[colorinlistoftodos,bordercolor=orange,backgroundcolor=orange!20,linecolor=orange,textsize=scriptsize]{todonotes}

\theoremstyle{plain}
\newtheorem{theorem}{Theorem}[section]
\newtheorem{lemma}[theorem]{Lemma}
\newtheorem{corollary}[theorem]{Corollary}
\newtheorem{proposition}[theorem]{Proposition}

\theoremstyle{definition}

\newtheorem{conjecture}[theorem]{Conjecture}

\newtheorem{question}[theorem]{Question}
\theoremstyle{remark}

\renewcommand{\le}{\leqslant}
\renewcommand{\leq}{\leqslant}
\renewcommand{\ge}{\geqslant}
\renewcommand{\geq}{\geqslant}

\newcommand{\R}{\mathbb{R}}

\def\P{\mathcal{P}}

\newcommand{\diame}{\mathsf{diam}}

\newcommand{\Ser}{\operatorname{Ser}}

\newcommand{\Par}{\operatorname{Par}}

\newcommand{\rk}{\operatorname{rk}}

\def \KG {\operatorname{KG}}
\def \SG {\operatorname{SG}}
\def \diame {\operatorname{diam}}

\def\calC{\mathcal{C}}

\def \B {\mathcal{B}}

\def \x {\boldsymbol{x}}

\begin{document}

\title{Matroids, intersecting bases, and Borsuk property}

\thanks{$\dagger$ Supported by  INSMI-CNRS}

\date{\today}

\author[G. L\'opez-Campos]{Gyivan L\'opez-Campos}
\address[Gyivan L\'opez-Campos]{Universidad Nacional Autonoma de M\'exico, Instituto de Matem\'aticas, Unidad Juriquilla and
IMAG, Univ. Montpellier, CNRS, Montpellier, France}
\email{gyivan.lopez@im.unam.mx}

\author[F. Meunier]{Fr\'ed\'eric Meunier}
\address[Frédéric Meunier]{CERMICS, ENPC, Institut Polytechnique de Paris, Marne-la-Vallée, France}
\email{frederic.meunier@enpc.fr}

\author[J.L. Ram\'{i}rez Alfons\'{i}n]{Jorge L. Ram\'{i}rez Alfons\'{i}n $^\dagger$}
\address[Jorge L. Ram\'{i}rez Alfons\'{i}n]{IMAG, Univ. Montpellier, CNRS, Montpellier, France}
\email{ jorge.ramirez-alfonsin@umontpellier.fr}


\begin{abstract} A subset $S$ of $\R^d$ has the Borsuk property if it can be decomposed into at most $d+1$ parts of diameter smaller than $S$. This is an important geometric property, inspired by a conjecture of Borsuk from the 1930s, which has attracted considerable attention over the years. In this paper, we define and investigate the Borsuk property for matroids, providing a purely combinatorial approach to the Borsuk property for matroid polytopes, a well-studied family of $(0,1)$-polytopes associated with matroids. We show that a sufficient condition for a matroid---and thus its matroid polytope---to have the Borsuk property is that the matroid or its dual has two disjoint bases. However, we show that this condition is not necessary by exhibiting infinite families of matroids having the Borsuk property and yet being such that every two bases intersect and every two cobases intersect. Kneser graphs, which form an important object from topological combinatorics, play a crucial role in most proofs.
\end{abstract}

\maketitle

\section{Introduction}\label{sec:intro}

The aim of this paper is to introduce and study the ``Borsuk property'' of matroids, taking a purely combinatorial approach to address the celebrated Borsuk conjecture. Formulated by Borsuk in 1933~\cite{borsuk1933drei}, this latter states that every bounded subset $S$ of $\R^d$ can be partitioned into at most $d+1$ parts, each of diameter smaller than that of $S$. A set $S$ admitting such a partition has the {\em Borsuk property}. This conjecture has attracted considerable attention, even after a counter-example was discovered by Kahn and Kalai in 1993~\cite{kahn1993counterexample}. Combinatorics has played a prominent role in its study and we aim at exploring further the relationship between combinatorics and this conjecture through the lens of matroids.

Consider a matroid $M$. As usual, the collection of its bases is denoted by $\B(M)$. The map $(B,B') \mapsto |B \triangle B'|$, where $\triangle$ denotes the symmetric difference, forms a distance on $\B(M)$ (actually, nothing else than the Hamming distance on the incidence vectors).  The {\em Borsuk number} of a matroid $M$, denoted by $f(M)$, is the minimal number of parts in a partition of $\B(M)$ where each part has diameter smaller than that of $\B(M)$. (``Diameter'' has to be understood according to the distance just defined.) When such a partition does not exist at all, i.e., when $M$ has exactly one basis, we set $f(M)\coloneqq +\infty$.

Denoting by $n$ the number of elements of $M$ and by $c$ the number of its connected components, we say that $M$ has the {\em Borsuk property} if $f(M)$ is at most $n-c+1$. (Notice that if $M$ has only one basis, then it has not the Borsuk property.) It is not too difficult to see that a matroid has the Borsuk property if and only if its matroid polytope has the Borsuk property in its affine hull(we will see that such a polytope is defined in $\R^n$ but its dimension is $n-c$). Our definition of the Borsuk property for matroids is thus consistent with the original geometric definition given above. The relations between these definitions, as well as possible consequences of our work in the original context, are discussed below in Section~\ref{sec:comb-geo}.

Our first contribution is the following.

\begin{theorem}\label{thm:main} If a matroid or its dual has two disjoint bases, then it has the Borsuk property.
\end{theorem}

As we will see, after introducing a suitable setting, the proof of Theorem~\ref{thm:main}, given in Section~\ref{sec:disj-borsuk}, essentially relies on upper bounds on  the chromatic number of ``Kneser graphs,'' which form an important family of graphs. This makes yet another connection with combinatorics. Studying generalized Kneser graphs allows us actually to get the exact value of the Borsuk number for specific families of matroid.

Combining Theorem~\ref{thm:main} with other combinatorial considerations will lead to the following theorem.

\begin{theorem}\label{thm:rank2}
Consider a matroid with at least two bases. If its rank is at most two, then it has the Borsuk property.
\end{theorem}

The proof essentially boils down to check finitely many situations. We establish the following theorem, which shows that it is possible to check in finite time whether there are matroids of a given rank failing to have the Borsuk property. But already for $r=3$, this is purely theoretical since the number of matroids to be considered would be too big.

\begin{theorem}\label{thm:rankr}
Consider a matroid of rank $r$ with at least two bases. If every two bases intersect, then its Borsuk number is upper bounded by $2^{r-2}(r+2)^2$.
\end{theorem}

Theorems~\ref{thm:main},~\ref{thm:rank2}, and~\ref{thm:rankr} are established in Section~\ref{sec:disj-borsuk}.

The following definition naturally arises from Theorem~\ref{thm:main}. A matroid has the {\em basis intersection property} if every two bases intersect. If a matroid and its dual has the basis intersection property, then we say that the matroid has the {\em strong basis intersection property}. Theorem~\ref{thm:main} can be alternatively stated: If a matroid does not have the strong basis intersection property, then it has the Borsuk property.

The results above show that it is thus still very possible that every matroid has the Borsuk property. We actually formulate this as a conjecture (Conjecture~\ref{conj:main}). Theorem~\ref{thm:main} shows that to disprove this conjecture, it is sufficient to focus on matroids that have the strong basis intersection property. Elementary arguments show that we can actually require connectivity (Corollary~\ref{cor:conn}). The behavior of connected matroids having the strong basis intersection property seems mysterious and, from our experience, they are not immediate to determine. In Section~\ref{sec:sbip}, we provide however many examples that rely on constructions keeping desirable properties, such as the classical series and parallel operations. The class of $\Theta_n$ matroids will also be considered.

\subsection*{Acknowledgments} We are grateful to Baptiste Gros for giving us a first example of matroid with the strong basis intersection property, which turned out to be the matroid $U_{2,3}\oplus U_{1,3}$ (see Section~\ref{sec:sbip}).

\section{Geometric connection:\\ Relation with the original Borsuk conjecture, and background}\label{sec:comb-geo}

\subsection{Relation with the Borsuk conjecture} The {\em Borsuk number} of a set $S\subset \R^d$ is the smallest number $f(S)$ of subsets in which it can be decomposed with each subset having diameter smaller than that of $S$. As shown by the following proposition, determining $f(P)$ for a polytope $P$ is essentially a (decidable) discrete problem. Even though this is a known fact for the community working on the Borsuk problem---at least from Larman's work (see the comments by Kahn and Kalai~\cite{kahn1993counterexample})---, we have not been able to find a proper reference from the literature. We provide a full proof for sake of completeness.

\begin{proposition}[Folklore]\label{prop:pol-vert}
For every polytope $P$, the equality $f(P) = f(V(P))$ holds (where $V(P)$ denotes the vertex set of $P$).
\end{proposition}

\begin{proof}
Consider a decomposition of $P$ into parts with diameter smaller than $\diame(P)$. The latter induces a decomposition of $V(P)$ with a non-larger number of parts and so that each part has diameter smaller than $\diame(P) = \diame(V(P))$ (the diameter of $P$ is attained by a pair of its vertices). Conversely, consider a decomposition of $V(P)$ into $k$ parts with diameter smaller than $\diame(V(P))$. Complete each part so as to get a decomposition of $P$ into $k$ parts, and such that if some part contains a vertex, it actually contains a whole neighborhood of it within $P$. Each part in this decomposition has diameter smaller than $\diame(P)$ (the diameter of $P$ is only attained by pairs of its vertices).
\end{proof}

We explain now that our definition of the Borsuk property for matroids is consistent with the original geometric definition. In that regard, we remind the definition of matroid polytopes. Let $M$ be a matroid on a ground set $E$. Denote by $\chi^S \in \{0,1\}^E$ the indicator vector of a subset $S$ of $E$, defined by 
\[
\chi^S_e \coloneqq \left \{ \begin{array}{rl} 1 & \text{if $e \in S$,} \\ 0 & \text{otherwise.} \end{array}\right.
\]
The {\em matroid polytope} of a matroid $M$, denoted by $P_M$, is the convex hull of the $\chi^B$ when $B$ ranges over $\B(M)$.  (Contrary to what is often assumed, we also consider this definition to be valid when $M$ has rank zero. In this case, $P_M$ is just a single point, namely the point $0$.)

\begin{proposition}\label{prop:numb}
The Borsuk number of a matroid equals the Borsuk number of its matroid polytope.
\end{proposition}

\begin{proof}
Let $M$ be a matroid. According to Proposition~\ref{prop:pol-vert}, the equality $f(P_M) = f(V(P_M))$ holds. Since $\|\chi^B - \chi^{B'}\|_2^2 = |B \triangle B'|$, the pair of vertices realizing the diameter of $V(P_M)$ correspond to the pair of bases realizing the diameter of $\B(M)$. We have thus $f(V(P_M)) = f(M)$.
\end{proof}

The original Borsuk property depends also on the dimension of the ambient space, for which it is clearly more interesting to choose the minimal possible value. For matroid polytopes, this dimension is known: $\dim(P_M) = n - c$, where $n$ is the number of elements of $M$ and $c$ the number of its connected components~\cite{oxley2011matroid}.

\begin{proposition}\label{prop:comb-geo}
A matroid $M$ has the Borsuk property if and only if its matroid polytope $P_M$ has the Borsuk property in $\R^{\dim(P_M)}$.
\end{proposition}

\begin{proof}
Immediate from the definitions and Proposition~\ref{prop:numb}.
\end{proof}

A consequence of Proposition~\ref{prop:comb-geo} is that any matroid with at least two bases and without the Borsuk property would provide a counter-example to the Borsuk conjecture. We emphasize that such a counter-example is not known yet.
 
\subsection{Background on the Borsuk conjecture}

Denote by $f(d)$ the smallest value such that every subset $S$ of $\R^d$ can be partitioned into at most $f(d)$ parts with diameter smaller than that of $S$. Borsuk's conjecture, which goes back to 1933, states that $f(d)=d+1$. The vertices of a regular $d$-simplex shows that $f(d)\geq d+1$ and it is known that $f(d)\leq d+1$ holds in several cases; see \cite{grunbaum1963borsuk,rauigorodskiui2004borsuk} and \cite[Chapter 31]{boltyanski1996excursions} for nice surveys. After sixty years, this outstanding conjecture was disproved by Kahn and Kalai~\cite{kahn1993counterexample} for $d=1325$ and for every $d \geq 2014$. The latter quantities have been improved and it is currently known to be false when $d \geq 64$ and true when $d \leq 3$. 

It is worth pointing out that all counter-examples on Borsuk's conjecture rely on purely combinatorial arguments. Erd\H{o}s~\cite{erdos1981my} and Larman~\cite{larman84} were among the first to suggest the construction of combinatorial counter-examples. All known non-trivial lower bounds for $f(d)$ have been obtained by considering two particular classes of integral polytopes: $(0,1)$-polytopes and $(0,1,-1)$-polytopes (or {\em cross-polytopes}). Estimations on the asymptotic behavior of $f(d)$ for both $(0,1)$-polytopes and cross-polytopes have been treated by Ra\u{\i}gorodski\u{\i}~\cite{Rai1,Rai4,raigorodskii2002borsuk}. By considering graph coloring problems and optimal binary codes, Ziegler~\cite[Theorem 1]{ziegler2001coloring} showed that every $(0,1)$-polytope of dimension $d \leq 9$ has the Borsuk property, i.e., can be decomposed into at most $d+1$ parts with diameter smaller than its diameter. Closely related results have also been obtained by Payan~\cite{payan1992chromatic}, Petersen~\cite{Peter}, and Schiller~\cite{Sch}.

\section{Notation and basic facts about matroids}\label{sec:matr}

We assume basic knowledge about matroids. In this section, we provide notation and properties of matroids that are maybe less known but that are important for our purpose. For proofs of these properties and complementary information about matroids, we refer the reader to the book by Oxley~\cite{oxley2011matroid}.

\subsection{Some notation}

The diameter of a set $X$, with respect to a suitable distance, is denoted by $\diame(X)$. The set of bases of a matroid $M$ is denoted by $\B(M)$. The {\em uniform matroid} whose bases are formed by all subsets of size $r$ of the ground set $[n]$ is denoted by $U_{r,n}$. The graphic matroid associated with a graph $G$ is denoted by $M(G)$.

\subsection{Connectivity and direct sum}
The {\em direct sum} of two matroids $M$ and $M'$, denoted by $M \oplus M'$, is the matroid whose ground set is the disjoint union of their ground sets and whose bases are the disjoint unions of their bases. The direct sum commutes with duality: $(M \oplus M')^\star = M^\star \oplus (M')^\star$ holds.

A matroid is {\em connected} if it cannot be written as the direct sum of two matroids. There is a unique way (up to isomorphism) to write a matroid $M$ as a direct sum of connected matroids. Each connected matroid in this sum is a {\em connected component} of $M$. Notice that by the commutation of direct sum and duality the number of connected components of a matroid is equal to the number of connected components of its dual.

\subsection{Series and parallel connections}
Consider two matroids $M$ and $M'$ that share exactly one element $p$. Denoting by $E$ and $E'$ the ground sets of respectively $M$ and $M'$, we have thus $E \cap E' = \{p\}$. Under extra condition on $p$, we can define two operations---the series connection and the parallel connection---that form a new matroid from $M$ and $M'$.

Assume that $p$ is a coloop of neither $M$ nor $M'$. Then the {\em series connection} of $M$ and $M'$, denoted by $\Ser(M,M')$, is the matroid with $E \cup E'$ as ground set and with \[
\bigl\{B \cup B' \colon B \in \B(M),\, B' \in \B(M), \, B \cap B' = \varnothing \bigl\}
\]as set of bases.

Assume that $p$ is a loop of neither $M$ nor $M'$. Then the {\em parallel connection} of $M$ and $M'$, denoted by $\Par(M,M')$, is the matroid with $E \cup E'$ as ground set and with 
\[
\begin{array}{l}
\bigl\{B \cup B' \colon B \in \B(M),\, B' \in \B(M'), \, B \cap B' = \{p\} \bigl\} \\[1.5ex] \cup \,\, \bigl\{(B \cup B') \setminus \{p\} \colon B \in \B(M),\, B' \in \B(M'), \, p \in (B \setminus B') \cup (B' \setminus B)\bigl\}
\end{array}
\] as set of bases.

It is easy to check that these two operations are dual of each other: the series connection of two matroids is the parallel connection of their duals, and vice versa. Formally, this writes
\begin{equation}\label{eq:ser-par}
\bigl(\Ser(M,M')\bigl)^\star = \Par(M^\star,(M')^\star) \, . 
\end{equation} Moreover, if $M$ and $M'$ are connected, then their series connection and their parallel connection are both connected as well~\cite[Proposition 7.1.17]{oxley2011matroid}.

\section{Preliminaries}\label{sec:prel}

In this section, we establish a few results about matroids, their Borsuk numbers, Kneser graphs, and related facts that will be useful for proving our main results. We also believe that some of them are of independent interest.

\subsection{Preliminaries on the Borsuk number}

\subsubsection{Duality}

The main operations involved in the definition of the Borsuk number and the Borsuk property (number of elements, distance, number of connected components, etc.) are invariant when we move to the dual. Hence, when it comes to the Borsuk number and the Borsuk property, a matroid behaves the same as its dual. The latter is formalized in the following proposition, whose proof is immediate---by remarking that the distance between two bases is equal to the distance between the dual bases---and thus omitted.

\begin{proposition}\label{prop:dual}
For every matroid $M$, the equality $f(M) = f(M^\star)$ holds. Moreover, $M$ has the Borsuk property if and only if $M^\star$ has the Borsuk property.
\end{proposition}

\subsubsection{Direct sum and connectivity}

The next proposition shows that to investigate upper bounds on the Borsuk number, it is relevant to start with connected matroids.

\begin{proposition}\label{prop:conn}
For every two matroids $M$ and $M'$, we have \[
f(M \oplus M') \le \min\bigl(f(M),f(M')\bigl) \, .
\] 
\end{proposition}

\begin{proof}
Since the matroids $M$ and $M'$ are arbitrary, it will be sufficient to show that $f(M \oplus M')\le f(M)$. Consider therefore a partition $\B_1,\B_2,\ldots,\B_k$ of $\B(M)$ into $k$ parts of diameter smaller than $\diame(\B(M))$. The partition $\B_1,\B_2,\ldots,\B_k$ induces a partition $\overline{\B}_1,\overline{\B}_2,\ldots,\overline{\B}_k$ of $M \oplus M'$ by setting $\overline{\B}_i$ to be the bases of $M \oplus M'$ of the form $B \cup B'$ with $B \in \B_i$ and $B' \in \B(M')$. 

We finish the proof by showing that no $\overline{\B}_i$ contain two bases of $M \oplus M'$ at distance $\diame(\B(M \oplus M'))$. Take two bases $B_1 \cup B_1'$ and $B_2 \cup B_2'$ of $M \oplus M'$ at distance $\diame(\B(M \oplus M'))$. We have
\[
\diame(\B(M \oplus M')) = |B_1 \triangle B_2| + |B_1' \triangle B_2'| \leq |B_1 \triangle B_2| + \diame(\B(M')) \, .
\]
Since $\diame(\B(M \oplus M')) = \diame(\B(M)) + \diame(\B(M'))$, the two bases $B_1$ and $B_2$ of $M$ are at distance $\diame(\B(M))$, which means that they cannot belong to the same $\B_i$, and therefore $B_1 \cup B_1'$ and $B_2 \cup B_2'$ cannot belong to the same $\overline{\B}_i$.
\end{proof}

We believe that the inequality in Proposition~\ref{prop:conn} is actually an equality. This is related to the celebrated Hedetniemi conjecture (now refuted); see Section~\ref{sec:open}.

The following corollary shows that to find matroids failing to have the Borsuk property, it is enough to focus on connected matroids.

\begin{corollary}\label{cor:conn}
Given two matroids, if one of them at least has the Borsuk property, then their direct sum has the Borsuk property as well.
\end{corollary}

\begin{proof}
Consider two matroids $M$ and $M'$, and suppose that $M$ has the Borsuk property. Denote by $n$ (resp.\,$n'$) the number of elements of $M$ (resp.\,$M'$) and by $c$ (resp.\,$c'$) its number of connected compontents. Proposition~\ref{prop:conn} implies that $f(M\oplus M') \leq n - c + 1$. Since the number of connected components in a matroid is at most the number of elements, we have $c' \leq n'$ and hence $f(M \oplus M') \leq n + n' - c - c' + 1$, as desired.
\end{proof}

\subsubsection{Series connection}

We end this section with two propositions about the Borsuk number and series connection.

\begin{proposition}\label{prop:ser}
For every two matroids $M$ and $M'$, we have \[
f\bigl(\Ser(M,M')\bigr) \leq f(M) + f(M') \, .
\] Moreover, if $M$ has two disjoint bases that do not contain $p$ (the element used for the series connection), then
\[
f\bigl(\Ser(M,M')\bigr) \leq \min\bigl(f(M),f(M')\bigr) \, .
\]
\end{proposition}

\begin{proof}
Let $\B_1,\B_2,\ldots,\B_k$ be a partition of $\B(M)$ into $k=f(M)$ parts of diameter smaller than $\diame(\B(M))$, and $\B'_1,\B'_2,\ldots,\B'_{k'}$ be a partition of $\B(M')$ into $k'=f(M')$ parts of diameter smaller than $\diame(\B(M'))$. \\

[First inequality] We build a partition $\overline{\B}_1,\overline{\B}_2,\ldots,\overline{\B}_{k+k'}$ of $\B(\Ser(M,M'))$ as follows. Consider a basis $B \cup B'$ of $\Ser(M,M')$. If $B'$ does not contain $p$, then put $B \cup B'$ in $\overline{\B}_i$ with $i$ such that $B \in \B_i$. If $B'$ contains $p$, then put $B \cup B'$ in $\overline{\B}_{k+i}$ with $i$ such that $B' \in \B'_i$. We check that no part $\overline{\B}_i$ contains two bases of $\Ser(M,M')$ at distance $\diame\bigl(\B(\Ser(M,M'))\bigr)$ apart. Consider two bases $B_1 \cup B_1'$ and $B_2 \cup B_2'$ of $\Ser(M,M')$ that are at distance $\diame\bigl(\B(\Ser(M,M'))\bigr)$ apart. We distinguish three cases.

Case 1) $p \notin B_1' \cup B_2'$. Then $B_1$ and $B_2$ are at distance $\diame(\B(M))$ apart (since otherwise we could increase the distance between the two bases $B_1 \cup B_1'$ and $B_2 \cup B_2'$). Thus $B_1$ and $B_2$ do not belong to the same $\B_i$, and $B_1 \cup B_1'$ and $B_2 \cup B_2'$ do not belong to the same $\overline{\B}_i$.

Case 2) $p \in B_1' \triangle B_2'$. Without loss of generality, we have $p \in B_1' \setminus B_2'$. Then $B_1 \cup B_1'$ belongs to some $\overline{\B}_i$ with $i > k$, and $B_2 \cup B_2'$ belongs to some $\overline{\B}_i$ with $i \leq k$.

Case 3) $p \in B_1' \cap B_2'$. Then $B_1'$ and $B_2'$ are at distance $\diame(\B(M'))$ apart (since otherwise we could increase the distance between the two bases $B_1 \cup B_1'$ and $B_2 \cup B_2'$). Thus $B_1'$ and $B_2'$ do not belong to the same $\B'_i$, and $B_1 \cup B_1'$ and $B_2 \cup B_2'$ do not belong to the same $\overline{\B}_i$.\\

[Second inequality] We build a partition $\overline{\B}_1,\overline{\B}_2,\ldots,\overline{\B}_{\min(k,k')}$ of $\B(\Ser(M,M'))$ as follows.

Consider a basis $B \cup B'$ of $\Ser(M,M')$. If $k \leq k'$, put $B \cup B'$ in $\overline{\B}_i$ with $i$ such that $B \in \B_i$. Otherwise, put $B \cup B'$ in $\overline{\B}_i$ with $i$ such that $B' \in \B'_i$. We check that no part $\overline{\B}_i$ contains two bases of $\Ser(M,M')$ at distance $\diame\bigl(\B(\Ser(M,M'))\bigr)$ apart. Consider two bases $B_1 \cup B_1'$ and $B_2\cup B_2'$ of $\Ser(M,M')$ that are at distance $\diame\bigl(\B(\Ser(M,M'))\bigr)$ apart. Notice that if $M$ contains two disjoint bases that do not contain $p$, then $\diame\bigl(\B(\Ser(M,M'))\bigr) = \diame(\B(M)) + \diame(\B(M'))$. This implies that $B_1$ and $B_2$ are necessarily at distance $\diame(\B(M))$ apart and that $B_1'$ and $B_2'$ are necessarily at distance $\diame(\B(M'))$ apart. Therefore, the bases $B_1 \cup B_1'$ and $B_2 \cup B_2'$ cannot belong to the same $\overline{\B}_i$.
\end{proof}

\subsection{Preliminaries on Kneser graphs and matroids}

We introduce the {\em matroid Kneser graph} associated with a matroid $M$, denoted by $\KG(M)$, as the graph with
vertices being the bases of $M$ and in which two vertices are joined by an edge if the intersection of the
corresponding bases is empty. (These graphs are actually special cases of the ``generalized Kneser graphs'' introduced by Dol'nikov~\cite{dol1988certain}.) The classical {\em Kneser graph} $\KG(n, r)$ is the particular case when
$M$ is $U_{r,n}$. These matroid Kneser graphs appear naturally in the proof of Theorem~\ref{thm:main}.

The key connection between matroid Kneser graphs and the Borsuk property when the matroid does not have the strong basis intersection property is given by the following.

\begin{lemma}\label{lem:kneser}
Consider a matroid $M$ with two disjoint bases. Then $\chi(\KG(M)) = f(M)$.
\end{lemma}

\begin{proof}
Since $M$ has two disjoint bases, the diameter of $\B(M)$ is equal to $2\rk(M)$, attained on pairs of disjoint bases. Consider a proper coloring of $\KG(M)$ with $k$ colors. Each color class contains only vertices associated with pairwise intersecting bases and is therefore of diameter smaller than $2\rk(M)$. Conversely, suppose there exists a partition of $\B(M)$ into $k$ parts of diameter smaller than $2\rk(M)$. None of these parts can contain two  disjoint bases. Using one color for each part provides a proper coloring of $\KG(M)$ with $k$ colors.
\end{proof}

Lemma~\ref{lem:kneser} shows that computing $f(M)$ for a matroid $M$ with two disjoint bases is equivalent to determining the chromatic number of $\KG(M)$. This motivates the determination of general upper bounds on $\KG(M)$.

\begin{proposition}\label{prop:upper}
Consider a matroid $M$ with $n$ elements. Then the following two properties hold.
\begin{enumerate}[label=\textup{(\roman*)}]
\item \label{coc} $\chi(\KG(M)) \leq \min_{C^\star\in\calC^\star(M)}|C^\star|$.\smallskip
\item \label{knes} If $n \geq 2 \rk(M) - 1$, then $\chi(\KG(M)) \leq  n - 2 \rk(M) +2$.
\end{enumerate}
\end{proposition}

The upper bound~\ref{knes} is nothing else than the classical upper bound on the chromatic number of Kneser graphs; see, e.g.,~\cite{lovasz1978kneser}.

\begin{proof}[Proof of Proposition~\ref{prop:upper}]
As explained just above, item~\ref{knes} is a consequence of classical results about Kneser graphs. Let us show item~\ref{coc}. Consider an arbitrary cocircuit $C^\star$. Color each basis with an arbitrary element from $B \cap C^\star$. Such an element exists since every basis and every cocircuit intersect in a matroid. This induces a coloring of $\KG(M)$. This coloring is proper since two disjoint bases are necessarily colored with distinct elements from $C^\star$.
\end{proof}

Proposition~\ref{prop:upper} does not provide tight upper bounds in general. For instance, consider the graph $G$ of Figure~\ref{fig:graphic}. Clearly, $\chi(\KG(M(G)))=1$ holds since there are two adjacent degree-two vertices and every two bases (spanning trees) intersect
On the other hand, item~\ref{coc} provides an upper bound of $2$ (there is no bridge in $G$, and any two edges incident to a degree-two vertex is a cocircuit). Item~\ref{knes} provides an upper bound larger than $2$, which is not tight either. The construction of Section~\ref{subsec:graphic} shows that this example can actually made arbitrarily large. (The chromatic number of Kneser graphs associated woth graphic matroids has been studied by Alishahi and Hajiabolhassan~\cite{alishahi2019chromatic}, who provide its exact value for many cases.)

We note however that~\ref{coc} provides better upper bounds than~\ref{knes}  in infinitely many cases. Consider for instance a rank-$3$ matroid $M$ with $n$ elements. If $M$ admits a line (hyperplane) with at least $h\geq 5 $ elements, then $\chi(\KG(M)) \le n-5$ by~\ref{coc}, while the upper bound~\ref{knes} is $n-4$ in that case. An example where this occurs is given by the matroid $V_h$ defined as follows: its bases are the triples of points that are not aligned in Figure~\ref{fig:vh}. 
In the next subsection, we will see rank-$3$ matroids where the bound~\ref{knes} is tight.

\begin{figure}
\begin{tikzpicture}[scale=0.6]

\coordinate (P1) at (0,0);
\coordinate (P2) at (1,1);
\coordinate (P3) at (2,2);
\coordinate (P4) at (3,3);
\coordinate (P5) at (4,4);
\coordinate (P6) at (4.8,4.8);
\coordinate (P7) at (6,6);
\coordinate (P22) at (-1,1);
\coordinate (P33) at (-2,2);
\coordinate (P44) at (-3,3);
\coordinate (P55) at (-4,4);
\coordinate (P66) at (-4.8,4.8);
\coordinate (P77) at (-6,6);

  \tikzset{vertex/.style={circle, fill=black, draw=black, minimum size=4pt, inner sep=0pt}}
  \node[vertex] (v1) at (P1) {};
  \node[vertex] (v2) at (P2) {};
  \node[vertex] (v3) at (P3) {};
  \node[vertex] (v4) at (P4) {};
  \node[vertex] (v5) at (P5) {};
  \node[vertex] (v22) at (P22) {};
  \node[vertex] (v33) at (P33) {};
  \node[vertex] (v44) at (P44) {};
  \node[vertex] (v55) at (P55) {};

\draw (P1) -- (P6);
\draw[dashed] (P6) -- (P7);
\draw (P1) -- (P66) ;
\draw[dashed] (P66) -- (P77);

\end{tikzpicture}
\caption{On each ``branch,'' there are $h$ points, not counting the point on the intersection. The triples of points that are not aligned form the bases of the matroid $V_h$.\label{fig:vh}}
\end{figure}
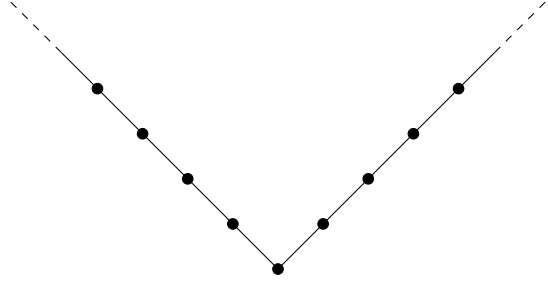


\subsection{Exploring further the chromatic number of matroid Kneser graphs}

Lemma~\ref{lem:kneser} shows that the chromatic number of the Kneser graph built from the bases of a matroid is relevant to study its Borsuk number. In this section, we compute this chromatic number for some classical matroids. Actually, since they all contain Schrijver graphs as induced subgraphs, this computation is immediate. We recall that the {\em Schrijver graph} $\SG(n,r)$ is the subgraph of $\KG(n,r)$ induced by the ``stable'' bases of $U_{r,n}$. (A basis $B$ is {\em stable} if it does not contain two cyclically consecutive elements from $[n]$, i.e., if $i$ and $j$ are distinct elements from $B$, then $2 \leq |i-j| \leq n-2$.) It is known~\cite{schrijver1978vertex} that $\chi(\SG(n,r)) = \chi(\KG(n,r))$.

The bases of the {\em Fano matroid} $F_7$ are the triples of points that are not aligned in Figure~\ref{F7}. Similarly, the bases of the {\em non-Pappus matroid} are the triples of points that are not aligned in Figure~\ref{non-Pappus}.

\begin{proposition}\label{prop:pappus} 
We have $\chi(\KG(F_7)) = 3$ and $\chi(\KG(\textup{non-Pappus matroid})) = 5$.
\end{proposition}

\begin{proof}
The graphs $\KG(F_7)$ and $\KG(\text{non-Pappus matroid})$ are respectively subgraphs of $\KG(7,3)$ and $\KG(9,3)$. This  shows that $3$ and $5$ are respective upper bounds on their chromatic numbers. With the labeling of Figure~\ref{fig:2}, it is immediate to check that $\SG(7, 3)$ and $\SG(9, 3)$ are subgraphs of respectively $\KG(F_7)$ and $\KG(\text{non-Pappus matroid})$. This shows that $3$ and $5$ are actually the exact values of their chromatic numbers.
\end{proof}

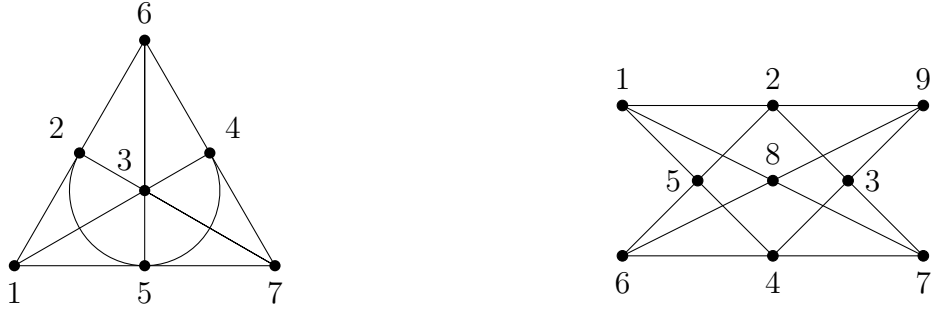
\begin{figure}

\begin{subfigure}{0.48\textwidth}
\centering
\begin{adjustbox}{trim=1cm 2.5cm 1cm 1cm,clip}
\begin{tikzpicture}[scale=2]

\coordinate (P1) at (0,0);                          
\coordinate (P2) at (-30:1);                        
\coordinate (P3) at (90:1);                         
\coordinate (P4) at (210:1);                        

\path[name path=LineA] ($(P1)!-1!(P2)$) -- ($(P1)!2!(P2)$); 
\path[name path=LineB] ($(P3)!-1!(P4)$) -- ($(P3)!2!(P4)$); 
\path[name path=LineC] ($(P1)!-1!(P4)$) -- ($(P1)!2!(P4)$); 
\path[name path=LineD] ($(P2)!-1!(P3)$) -- ($(P2)!2!(P3)$); 
\path[name path=LineE] ($(P1)!-1!(P3)$) -- ($(P1)!2!(P3)$); 
\path[name path=LineF] ($(P2)!-1!(P4)$) -- ($(P2)!2!(P4)$); 

\path[name intersections={of=LineA and LineB, by=I}];
\path[name intersections={of=LineC and LineD, by=J}];
\path[name intersections={of=LineE and LineF, by=K}];

\node[circle,fill=black,inner sep=1.5pt,label={[xshift=0.4mm, yshift=0.7mm]above left:$3$}] at (P1) {};
\node[circle,fill=black,inner sep=1.5pt,label=below:$7$] at (P2) {};
\node[circle,fill=black,inner sep=1.5pt,label=above:$6$] at (P3) {};
\node[circle,fill=black,inner sep=1.5pt,label=below:$1$] at (P4) {};

\coordinate (P5) at (I);
\coordinate (P6) at (J);
\coordinate (P7) at (K);
\node[circle,fill=black,inner sep=1.5pt,label=above left:$2$] at (P5) {};
\node[circle,fill=black,inner sep=1.5pt,label=above right:$4$] at (P6) {};
\node[circle,fill=black,inner sep=1.5pt,label=below:$5$] at (P7) {};

\draw (P1) -- (P2) -- (P5);             
\draw (P4) -- (P1) -- (P6);             
\draw (P1) -- (P3) -- (P7);             
\draw (P4) -- (P7) -- (P2);             
\draw (P2) -- (P6) -- (P3);             
\draw (P3) -- (P5) -- (P4);             


\draw 
  let 
    \p1 = (P1),
    \p2 = (P5),
    \n1 = {veclen(\x2-\x1,\y2-\y1)},
  in
    (P5) arc [start angle=150, end angle=390, radius={\n1}];

\end{tikzpicture}
\end{adjustbox}
\caption{The Fano matroid: the bases are the triples of points that are not aligned.\label{F7}}
\end{subfigure}
\hfill
\begin{subfigure}{0.48\textwidth}
\centering
\begin{tikzpicture}[scale=1]

\coordinate (P1) at (-2,1);
\coordinate (P2) at (0,1);
\coordinate (P9) at (2,1);
\coordinate (P5) at (-1,0);
\coordinate (P8) at (0,0);
\coordinate (P3) at (1,0);
\coordinate (P6) at (-2,-1);
\coordinate (P4) at (0,-1);
\coordinate (P7) at (2,-1);

  \tikzset{vertex/.style={circle, fill=black, draw=black, minimum size=4pt, inner sep=0pt}}
  \node[vertex, label=$1$] (v1) at (P1) {};
  \node[vertex, label=$2$] (v2) at (P2) {};
  \node[vertex, label=right:$3$] (v3) at (P3) {};
  \node[vertex, label=below:$4$] (v4) at (P4) {};
  \node[vertex, label=left:$5$] (v5) at (P5) {};
  \node[vertex, label=below:$6$] (v6) at (P6) {};
  \node[vertex, label=below:$7$] (v7) at (P7) {};
  \node[vertex, label=$8$] (v8) at (P8) {};
  \node[vertex, label=$9$] (v9) at (P9) {};

\draw (P1) -- (P9);
\draw (P1) -- (P7);
\draw (P1) -- (P4);
\draw (P6) -- (P2);
\draw (P6) -- (P9);
\draw (P6) -- (P7);
\draw (P4) -- (P9);
\draw (P2) -- (P7);
\end{tikzpicture}

\caption{The non-Pappus matroid: the bases are the triples of points that are not aligned.\label{non-Pappus}}
\end{subfigure}
\caption{Two rank-$3$ matroids whose Kneser graph has its chromatic number determined by~\ref{knes} in Proposition~\ref{prop:upper}.\label{fig:2}}
\end{figure}

Let us consider {\em lattice paths} in $\mathbb{Z}^2$, starting at~$(0,0)$ and ending at
$(m,r)$, with $(1,0)$ and $(0,1)$ steps. A $(1,0)$ step is an {\em East step}, denoted by $E$, and a $(0,1)$ step is a {\em North step}, denoted by $N$. Given two lattice paths $P,Q$ such that $P$ never goes below $Q$, we denote by $\P$ the set of all lattice paths from $(0,0)$ to $(m,r)$ that go neither above~$P$ nor below $Q$. It turns out that the set
\[
\bigl\{B \subseteq [m+r] \colon B \text{ is the North steps of a path in }\P\bigr\}
\]
is the set of bases of a matroid, called the {\em lattice path matroid} and denoted by $M[P,Q]$. This subclass of transversal matroids has been introduced by Bonin, de Mier, and Noy~\cite{bonin2003lattice}. Three special cases are worth highlighting. For the last two cases, we assume further that $m \geq r$.
\begin{itemize}
\item When
$P=(\underbrace{N\dots N}_{r}\underbrace{E\dots E}_{m})$ and $Q=(\underbrace{E\dots E}_{m}\underbrace{N\dots N}_{r})$, the matroid $M[P,Q]$ is actually $U_{r,m+r}$.\\[-2ex]
\item When $P=(\underbrace{NE\dots NE}_{2r}\underbrace{E\dots E}_{m-r})$ and $Q=(\underbrace{E\dots E}_{m}\underbrace{N\dots N}_{r})$, the matroid $M[P,Q]$ is the so-called {\em Catalan matroid}, denoted by $C(r,m)$.\\[-2ex]
\item When $P=(\underbrace{NE\dots NE}_{2r}\underbrace{E\dots E}_{m-r})$ and $Q=(\underbrace{E\dots E}_{m-r}\underbrace{EN\dots EN}_{2r})$, the matroid $M[P,Q]$ is denoted by $C^-(r,m)$.
\end{itemize}

\begin{figure}[h]
\begin{center}
 \includegraphics[width=.7\textwidth]{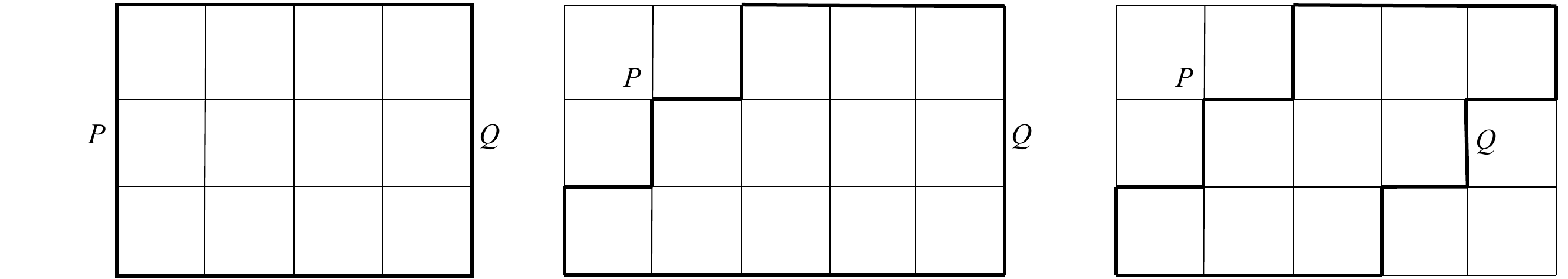}
 \caption{$U_{3,7}$ (left), $C(3,5)$ (middle) and $C^-(3,5)$ (right).}\label{fig:latticepaths}
 \end{center}
\end{figure}

\begin{proposition} Let $m\ge r\ge 1$ be integers. Then,
\[
\chi(\KG(C(r,m))) = \chi(\KG(C^-(r,m))) = m-r+2\, .
\]
\end{proposition}

\begin{proof}  It can be checked that $\SG(m+r,r)$ is an  subgraph of $\KG(C^-(r,m))$, itself being an subgraph of $\KG(C(r,m))$, itself being an subgraph of $\KG(m+r,r)$. The result follows from $\chi(\SG(m+r,r)) = \chi(\KG(m+r,r)) = m+r- 2r+2=m-r+2$.
\end{proof}

\section{Disjoint bases, disjoint cobases, and Borsuk property}
 
\label{sec:disj-borsuk}

\subsection{Proof of Theorem \ref{thm:main}} 

We now use the ingredients given in Section~\ref{sec:prel} to prove Theorem~\ref{thm:main}.

\begin{proof}[Proof of Theorem~\ref{thm:main}]
Let $M$ be a matroid. Denote by $n$ the size of its ground set and by $r$ its rank. By Proposition~\ref{prop:dual}, we can assume without loss of generality that $M$ has two disjoint bases. Write $M$ as the direct sum of its connected components $M = M_1 \oplus M_2 \oplus \cdots \oplus M_c$, and denote by $n_i$ the number of elements of $M_i$ and by $r_i$ its rank. Each $M_i$ also admits two disjoint bases. Every proper coloring of $\KG(n_i,r_i)$ induces a proper coloring of $\KG(M_i)$ because every basis of $M_i$ is a vertex of $\KG(n_i,r_i)$. Proposition~\ref{prop:upper} shows that $\chi(\KG(M_i))\leq n_i$. With Lemma~\ref{lem:kneser}, we get $f(M_i) \leq n_i$, which means that each $M_i$ has the Borsuk property. The latter implies, by Corollary~\ref{cor:conn}, that $M$ has the Borsuk property.
\end{proof}

\subsection{Matroids of rank at most two}\label{subsec:m2}

In this section, we prove Theorem~\ref{thm:rank2}, which states that every matroid of rank at most two and with at least two bases has the Borsuk property.
 
We start with a preliminary result.

\begin{proposition}\label{prop:rank2}
Consider a rank-two matroid with at least two bases. If it has the strong basis intersection property, then it has a loop or a coloop.
\end{proposition}

\begin{proof}
Let $[n]$ be the ground set of the matroid for some integer $n$, and assume that it has the strong basis intersection property.

Without loss of generality, let $\{1,2\}$ and $\{1,3\}$ be two bases of the matroid. Since every two cobases intersect, $n$ is at least $4$. We proceed by contradiction. Suppose that the matroid has no loop and no coloop. There exists thus a basis containing $4$, which is necessarily $\{1,4\}$ as this basis must intersect the two other bases. The element $1$ not being a coloop, there must be a basis missing it. But such a basis would then be disjoint from at least one of  $\{1,2\}$, $\{1,3\}$, and $\{1,4\}$, which is not possible.
\end{proof}

The proof considers implicitly the matroid with $\{1,2,3,4\}$ as ground set, and with three bases: $\{1,2\}$, $\{1,3\}$, and $\{2,3\}$. This matroid has rank two and has the strong basis intersection property. With the help of Proposition~\ref{prop:rank2}, we can prove Theorem~\ref{thm:rank2}, which shows in particular that even though this matroid has the strong basis intersection property, it has the Borsuk property. This matroid is actually the smallest example with at least two bases showing that not to have the strong basis intersection property is not a necessary condition to have the Borsuk property. (Formally, a matroid formed only by a single loop is already such an example, but the reader will certainly agree that it is not really an interesting one.) Infinite families of other examples will be given in Section~\ref{sec:sbip}.

\begin{proof}[Proof of Theorem~\ref{thm:rank2}]
Consider a matroid $M$ with at least two bases. It cannot have rank zero. Suppose that $M$ has rank one. It has the Borsuk property: it is immediate to check that every element is either a basis, or a loop; by Corollary~\ref{cor:conn}, we can assume that it has no loop; by Theorem~\ref{thm:main}, we can further assume that $M$ has only one element, and the conclusion follows. 

Finally, suppose that $M$ has rank two. Write $M$ as the direct sum of its connected components.By Corollary~\ref{cor:conn}, we are left with the case where $M$ is a rank-two connected matroid, and Proposition~\ref{prop:rank2} shows then that $M$ does not have the strong basis intersection property, and we can apply Theorem~\ref{thm:main}.
\end{proof}

\subsection{Matroids of rank larger than two}

In this section, we prove Theorem~\ref{thm:rankr}, which shows that a matroid with basis intersection property has a Borsuk number bounded by an expression depending only on the rank.

\begin{proof}[Proof of Theorem~\ref{thm:rankr}]
By Proposition~\ref{prop:conn}, we can assume that $M$ is connected. Denote by $s$ the minimal size of the intersection of any two bases. Pick an arbitrary basis $B_0$. Then for each element $e$ in $B_0$, pick a basis $B(e)$ of $M$ not containing $e$. Such a basis exists since $M$ is connected. Now, using these $B_0$ and $B(e)$’s, we will assign to each basis $B$ a $(s+1)$-tuple of elements in $B \cap (B_0 \cup B(e))$ for some $e$ depending on $B$. 

We proceed as follows. Pick arbitrarily $s$ elements in $B \cap B_0$. This forms an $s$-tuple that we denote by $X$. Pick also an arbitrary element $e$ in $X$. The intersection $B \cap B(e)$ contains at least one element not in $X$ since this intersection has size at least $s$ and does not contain $e$. Denote by $f$ such an element in $B \cap B(e)$. We assign to $B$ the $(s+1)$-tuple formed by $X$ and $f$. We emphasize that $X$ and $f$ are included in $B$.

Now, notice that any two bases with intersection of size exactly $s$ is assigned distinct $(s+1)$-tuples since every such $(s+1)$-tuple is included in the basis to which it is assigned. Consider the partition where each part is formed by the bases assigned to a $(s+1)$-tuple. In this partition, any two bases of $M$ that form a diameter are associated with distinct parts.

We finish by computing an upper bound on the number of parts. We can choose $B_0$ and $B(\bar e)$ for some $\bar e$ such that their intersection is of size $s$. The basis $B(\bar e)$ is already fine for every $e$ not in the intersection $B_0 \cap B(\bar e)$. So, we only need to choose $B(e)$’s for $e$ in this intersection. In total, these $B(e)$ (with $B(\bar e)$) cover at most $r-s + s(r-s)$ elements outside $B_0$, which is smaller than $\frac 1 4 (r+1)^2$. The number of possible $(k+1)$-tuples is thus upper-bounded by 
\[
{r \choose s} \times \left(r+\frac 1 4 (r+1)^2 \right) \leq 2^{r-2} (r+2)^2 \, . \qedhere
\] 
\end{proof}

\section{Matroids with the strong basis intersection property}\label{sec:sbip}

The strong basis intersection property, introduced in Section~\ref{sec:intro}, finds its motivation in Theorem~\ref{thm:main}: Not to have the strong basis intersection property implies to have the Borsuk property. Even if this condition is not necessary---see the discussion in Section~\ref{subsec:m2}---Theorem~\ref{thm:main} and Corollary~\ref{cor:conn} show that if we want to find a matroid that does not have the Borsuk property---something which is still possible even if we do not know any such matroid---, it is sufficient to focus on connected matroids with the strong basis intersection property. Since we think that this question is also interesting on its own, we focus in this section on considering some connected matroids that have the strong basis intersection property. Even though there are many such matroids, building them is not a straightforward task.

Notice that if we remove the connectivity condition, it is easy to build matroids that have the strong basis intersection property from matroids that just have the basis intersection property. (We remind the reader that a matroid has the basis intersection property if every two bases intersect, but not necessarily every two cobases.) It is immediate that if a matroid $M$ has the basis intersection property, then $M \oplus M^\star$ has the strong basis intersection property. This shows for instance that $U_{2,3} \oplus U_{1,3}$ has the strong basis intersection property.

Among the classical families of matroids, $\Theta_n$ matroids are connected, have the strong basis intersection property, and satisfy the Borsuk property for $n\geq 2$ (Proposition~\ref{prop:theta_n} below). In this section, in addition of proving Proposition~\ref{prop:theta_n}, we discuss how the classical series and parallel operations for matroids can be used to build arbitrarily large connected matroids with the strong basis intersection property (Section~\ref{subsec:ser-par}). Finally, we also describe a construction allowing to obtain arbitrarily large connected graphic matroids with the strong basis intersection property (Section~\ref{subsec:graphic}).

\subsection{$\Theta_n$ matroids}\label{subsec:theta_n}

For $n\geq 1$, the $\Theta_n$ matroid has its ground set formed by two disjoint sets of the form $X=\{x_1,\ldots ,x_n\}$ and $Y=\{y_1,\ldots ,  y_n\}$ and its bases are the subsets $B \subseteq X \cup Y$ of size $n$ such that $|B \cap X | \leq 2$ and $B \neq (Y - y_i) \cup x_i$ for all $i$.

\begin{proposition}\label{prop:theta_n}
For $n\geq 5$, the matroid $\Theta_n$ has the strong basis intersection property and its Borsuk number is at most $n-2$.
\end{proposition}

\begin{proof}
It is immediate to check that it has the basis intersection property. This implies that it has actually the strong basis intersection property since the $\Theta_n$ matroid is self-dual~\cite[p.~663]{oxley2011matroid}. Two bases $B,B'$ realize the diameter of $\B(\Theta_n)$ only if 
\begin{equation}\label{eq:theta_n}
B\cap B'\cap X = \varnothing \quad \mbox{and} \quad |B \cap X| = B' \cap X| = 2 \, .
\end{equation} For $i \in [n-2]$, set $\B_i$ as the set of bases $B$ such that $i=\min\bigl(n-2,\min\{i \in [n] \colon x_i \in B\cap X\}\bigr)$. Since any two bases satisfying~\eqref{eq:theta_n} are assigned to distinct parts $\B_i$, this forms a partition of $\B(\Theta_n)$ such that each part has a diameter smaller than that of $\B(\Theta_n)$.
\end{proof}

It is not difficult to establish the following equalities:
\[
f(\Theta_1) = +\infty \quad \mbox{and} \quad f(\Theta_2) = f(\Theta_3) = f(\Theta_4) = 2\, .
\]
Therefore, the $\Theta_n$ matroid has the Borsuk property for every $n\geq 2$.

\subsection{Series and parallel operations}\label{subsec:ser-par}
We establish the following result implying that infinite families of connected matroids with the strong basis intersection property can be easily built with the help of series and parallel connections.

\begin{lemma}\label{lem:connec}
If $M$ or $M'$ has the basis intersection property, then $\Ser(M,M')$ has the basis intersection property. If every two bases of $M$ or $M'$ intersect in at least two elements, then $\Par(M,M')$ has the basis intersection property.
\end{lemma}

\begin{proof}
The first implication is immediate. Consider the second one. Assume that every two bases of $M$ intersect in at least two elements. Consider two bases $B$ and $\tilde B$ of $\Par(M,M')$. There exist $B_1$ and $B_2$  bases of $M$ and $B_1'$ and $B_2'$ bases of $M'$ such that $(B_1 \cup B_1')\setminus \{p\}$ is included in $B$ and $(B_2 \cup B_2')\setminus \{p\}$ is included in $\tilde B$. (These inclusions can actually be equalities.) We have $(B_1 \cap B_2) \setminus\{p\} \subseteq B \cap \tilde B$. By the assumption on the bases of $M$, the set $(B_1 \cap B_2) \setminus\{p\}$ is non-empty, which implies that the $B$ and $\tilde B$ share an element.
\end{proof}

From Lemma~\ref{lem:connec}, we obtain the following.

\begin{proposition}\label{prop:sbip-ser}
Let $M$ be a matroid such that every two bases intersect in at least two elements. Then $\Ser(M,M^\star)$ has the strong basis intersection property.
\end{proposition}

\begin{proof}
By the hypothesis on $M$ and Lemma~\ref{lem:connec}, $\Ser(M,M^\star)$ and $\Par(M,M^\star)$ have each the basis intersection property. By~\eqref{eq:ser-par}, every two cobases of $\Ser(M,M^\star)$ intersect. The matroid $\Ser(M,M^\star)$ has thus the strong basis intersection property.
\end{proof}

If $M$ is connected, then so is $\Ser(M,M^\star)$. With Proposition~\ref{prop:sbip-ser}, arbitrarily large connected matroids with the strong basis intersection property can thus be built. For instance, if $M$ is the uniform matroid $U_{m,m+n}$ with $m \geq n+2$, then it satisfies the condition of Proposition~\ref{prop:sbip-ser}, and $\Ser(U_{m,m+n},U_{n,m+n})$ is a connected matroid satisfying the strong basis intersection property. By Proposition~\ref{prop:ser}, it has the Borsuk property.

\subsection{Graphic matroids with the strong basis intersection property}\label{subsec:graphic}

Consider a graph $H$ with a marked vertex $v_0$ and a graph $L$ with at least $\deg_H(v_0)$ vertices. Build a new graph as follows: 
\begin{itemize}
\item replace $v_0$ with $L$.
\item make sure that each edge of $H$ incident to $v_0$ is now incident to a distinct vertex in $L$.
\end{itemize}
An example obtained with $H=K_3$ and $L$ a planar graph with $6$ vertices is illustrated on Figure~\ref{fig:graphic}.

\begin{proposition}\label{prop:sbip-graphic}
If $M(H)$ and $M^\star(L)$ both have the basis intersection property, then $M(G)$ has the strong basis intersection property.
\end{proposition}

Moreover, if $H$ and $L$ are $2$-connected, then so is $G$. Proposition~\ref{prop:sbip-graphic} combined with this remark allows to build arbitrarily large connected graphic matroids with the strong basis intersection property. Figure~\ref{fig:graphic} provides an example of a graph $G$ built this way that is $2$-connected and planar. This implies in particular that $M(G)$ is connected and that its dual is graphic.


\begin{proof}[Proof of Proposition~\ref{prop:sbip-graphic}]
Suppose that $M(H)$ and $M^\star(L)$ both have the basis intersection property.

Consider two bases $B_1,B_2$ of $M(G)$. By contracting each $L$ to a single vertex, we transform $G$ into $H$. Each of $B_1$ and $B_2$ becomes a covering spanning subgraph of $H$ and contains a basis of $M(H)$. Thus $B_1$ and $B_2$ have a non-empty intersection.

Consider now two cobases $B_1^\star$ and $B_2^\star$ of $M(G)$. Each of $B_1^\star$ and $B_2^\star$ induces the complement of a forest on $L$ and contains thus a cobasis of $L$. Thus $B_1^\star$ and $B_2^\star$ have a non-empty intersection. 
\end{proof}

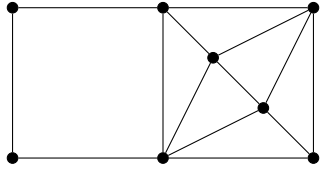
\begin{figure}
 \begin{tikzpicture}[scale=2]
  \tikzset{vertex/.style={circle, fill=black, draw=black, minimum size=4pt, inner sep=0pt}}
  \node[vertex] (v1) at (0,0) {};
  \node[vertex]  (v2) at (0,1) {};
  \node[vertex]  (v3) at (1,1) {};
  \node[vertex]  (v4) at (2,1) {};
  \node[vertex]  (v5) at (2,0) {};
  \node[vertex]  (v6) at (1,0) {};
  
  \coordinate (w1) at ($ (v3) !0.333! (v5) $);
  \coordinate (w2) at ($ (v3) !0.667! (v5) $);
  
  \node[vertex]  (v7) at (w1) {};
  \node[vertex]  (v8) at (w2) {};
  
  \draw (v1) -- (v2) -- (v3) -- (v4) -- (v5) -- (v6) -- (v1);
  \draw (v3) -- (v6);
  \draw (v3) -- (v7) -- (v8) -- (v5);
  \draw (v8) -- (v6) -- (v7) -- (v4) -- (v8);

\end{tikzpicture}

\caption{Planar graph $G$ obtained with the construction of Section~\ref{subsec:graphic}, with $H=K_3$ and $L$ a $6$-vertex planar graph such that $M^\star(L)$ has the basis intersection property. The graphic matroid $M(G)$ is connected and has the strong basis intersection property. It is moreover not difficult to check that $f(M(G))=3$, which shows that $M(G)$ has the Borsuk property.\label{fig:graphic}}
\end{figure}

The following proposition shows that to get an example of a graph $G$ built as above without the Borsuk property, it is enough to consider graphs $L$ such that $M(L)$ does not have the basis intersection property. 

\begin{proposition}\label{prop:graphic-bn}
If $M(L)$ does not have the basis intersection property, then $f(M(G)) \leq f(M(H))$.
\end{proposition}

\begin{proof}
Suppose that $M(L)$ does not have the basis intersection property, i.e., that their exist in $L$ two disjoint spanning trees, and denote by $t(H)$ the minimum number of edges shared by two spanning trees of $H$.

Using two spanning trees of $H$ sharing exactly $t(H)$ edges and two disjoint spanning trees of $L$, it is straightforward to build two spanning trees of $G$ sharing exactly $t(H)$ edges. Hence, two spanning trees of $G$ realizing the diameter of $\B(M(G))$ share at most $t(H)$ edges. (It is not difficult to see that they actually share exactly that number of edges, but this is not needed for the proof.)

Given a partition $\B_1,\B_2,\ldots,\B_k$ of $\B(M(H))$ into $k$ parts of diameter smaller than that of $\B(M(H))$, we build a partition $\overline\B_1,\overline\B_2,\ldots,\overline\B_k$ of $\B(M(G))$ into $k$ parts as follows. For each spanning tree $T$ of $G$, pick a ``representative'' spanning tree $T'$ of $H$ contained in $T$, and put $T$ in $\overline\B_i$ for $i$ such that $\B_i$ contains the representative $T'$. Two spanning trees of $G$ in the same $\overline\B_i$ contain by definition two spanning trees of $H$ at distance less than $\diame\bigl(\B(M(H))\bigr)$, i.e., two spanning trees of $H$ sharing at least $t(H)+1$ edges. Two spanning trees of $G$ in the same $\overline\B_i$ share therefore at least $t(H)+1$ edges and, as noted above, this implies that they cannot realize the diameter of $\B(M(G))$. This shows that the inequality $f(M(G)) \leq f(M(H))$ holds.
\end{proof}

Proposition~\ref{prop:graphic-bn} shows that the Borsuk number of $M(G)$, for the graph $G$ of Figure~\ref{fig:graphic}, is at most $3$, and an easy case-checking shows that it cannot be $2$ or less.

\section{Open questions}\label{sec:open}

We think that the main open question is whether every matroid with at least two bases has the Borsuk property. According to Section~\ref{sec:comb-geo}, this is equivalent to the question whether there exist matroid polytopes that are counter-examples to the Borsuk conjecture. The answer to this question is unknown.

\begin{conjecture}\label{conj:main}
Every matroid with at least two bases has the Borsuk property.
\end{conjecture}

Another question that arises naturally is whether the inequality in Proposition~\ref{prop:conn} can be strict.

\begin{question}\label{quest:hedet}
Do there exist matroids $M$ and $M'$ such that $f(M \oplus M') < \min(f(M),f(M'))$?
\end{question}

This question is actually related the celebrated Hedetniemi conjecture~\cite{hedetniemi1966homomorphisms}, which has been refuted by Shitov in 2019~\cite{shitov2019counterexamples}. It deals with the chromatic number of the {\em categorical product} of two graphs $G=(V,E)$ and $G'=(V',E')$, denoted by $G \times G'$ and defined as follows. The vertex set of $G \times G'$ is the set of all ordered pairs $(v,v') \in V \times V'$. Edges are formed by the unordered pairs $(u,u')(v,v')$ such that $uv \in E$ and $u'v' \in E'$.

Hedetniemi's conjecture states that $\chi(G \times H) = \min(\chi(G),\chi(H))$. Noticing that
\[
\KG(M \oplus M') = \KG(M) \times \KG(M') \, ,
\] Lemma~\ref{lem:kneser} shows that a positive answer to Question~\ref{quest:hedet} for matroids with disjoint bases would build a counter-example to Hedetniemi conjecture with matroid Kneser graphs. No such example is known so far. Note that the symbol `$\times$' is used in this paragraph with two different meanings: it represents the categorical product of graphs, but also the usual Cartesian product of sets (when we consider the product of the vertex sets). In the next paragraph, it is again used with this latter meaning.

By the equality $P_{M \oplus M'} = P_M \times P_{M'}$, Question~\ref{quest:hedet} is also related to whether the Borsuk property is kept when we perform the Cartesian product of two matroid polytopes having the Borsuk property. As far as we are aware, this question, even for general polytopes, has not been studied up to now.


\bibliographystyle{amsplain}
\bibliography{bibliography}

\end{document}